\UseRawInputEncoding
\documentclass[leqno]{amsart}
\usepackage{stmaryrd,graphicx}
\usepackage{amssymb,mathrsfs,amsmath,amscd,amsthm,color}
\usepackage{float}
\usepackage[all,cmtip]{xy}
\DeclareMathAlphabet{\mathpzc}{OT1}{pzc}{m}{it}
\usepackage{amsfonts,latexsym,wasysym}
\newcommand{\im}{\operatorname{Im}}
\newcommand{\tc}{\mathbf{TC}}
\newcommand{\cat}{\mathbf{cat}}
\newcommand{\rk}{\mathbf{wrk}}

\renewcommand{\int}{\operatorname{int}}
\newcommand{\ev}{\operatorname{ev}}

\newcommand{\Ker}{\operatorname{Ker}}

\newcommand{\ui}{[0,1]}

\newcommand{\fw}{\operatorname{w}}
\newcommand{\sw}{\bigcurlyvee}

\newcommand{\bbn}{\mathbb{N}}

\newcommand{\bbr}{\mathbb{R}}

\newcommand{\ov}{\overline}

\newtheorem{theorem}{Theorem}[section]
\newtheorem{lemma}[theorem]{Lemma}
\newtheorem{proposition}[theorem]{Proposition}

\newtheorem{corollary}[theorem]{Corollary}
\theoremstyle{definition}\newtheorem{definition}[theorem]{Definition}
\newtheorem{example}[theorem]{Example}

\newtheorem{remark}[theorem]{Remark}

\title{Motion Planning on One-Dimensional Peano Continua}

\author[J. Brazas]{Jeremy Brazas}
\address{West Chester University\\ Department of Mathematics\\
West Chester, PA 19383, USA}
\email{jbrazas@wcupa.edu}

\author[P. Pave\v si\' c]{Petar Pave\v si\' c}
\address{Faculty of Mathematics and Physics, University of Ljubljana and Institute of Mathematics, Physics and Mechanics, Ljubljana}
\email{petar.pavesic@fmf.uni-lj.si}
\thanks{The authors are grateful to the referees for many helpful suggestions and the vital correction to Theorem \ref{mainthm4}, which originally claimed a single value in the middle case. This research was supported by the Slovenian Research Agency program P1-0292 and grant J1-4001.}

\begin{document}

\begin{abstract}
We study the Lusternik-Schnirelmann category and topological complexity of 1-dimensional spaces.
We define both invariants as lengths of suitable closed filtrations, as opposed to a more common
definition based on open covers. Our main results provide a precise description of $\cat(X)$ and 
$\tc(X)$ for certain 1-dimensional Peano continua $X$ in terms of the wildness rank of $X$. 
A surprising consequence is that $\cat(X)$ and $\tc(X)$ of a general 
1-dimensional space $X$ can be arbitrarily high, which is in stark contrast with 
the analogous results for 1-dimensional CW-complexes.
\end{abstract}

\maketitle

%=====================================================================================================
\section{Introduction}\label{sec:Introduction}
%=====================================================================================================

Is there a continuous rule that, to every pair of points $x,x'$ in a space $X$, assigns a path in $X$
starting at $x$ and ending at $x'$? This question, motivated by applications in robot motion planning,
was a starting point for Michael Farber's theory of topological complexity. Farber observed that a
continuous rule valid for all pairs $(x,x')\in X\times X$ exists if and only if $X$ is 
contractible. He thus defined the topological complexity $\tc(X)$ of motion planning in $X$ as the 
minimal integer $n$ for which
there exists an open cover $U_0,\ldots,U_n$ of $X\times X$ such that each $U_i$ admits a continuous 
motion plan, i.e. a continuous choice of a path connecting $x$ and $x'$ for every $(x,x')\in U_i$. (Note the 
so-called \emph{normalized} indexing with respect to which a contractible space $X$ has $\tc(X)=0$.) Farber 
proved that $\tc(X)\le 2\dim(X)$ if $X$ is a connected CW-complex. However, if $X$ is not locally 
simply-connected, then the requirement that the domains $U_i$ in the covering of $X\times X$ are open implies 
that  $\tc(X)$ is infinite. Thus, in order to discuss motion planning on more general spaces 
like fractals or 1-dimensional continua, we are led to work with an alternative definition of topological
complexity  
based on closed filtrations. This definition yields the same value of $\tc(X)$ for CW-complexes
but is more suitable for spaces with intricate local features. The resulting theory is very satisfactory and we were able to obtain precise values for the topological complexity of many 1-dimensional Peano continua.

To describe our results, we first recall that a point of $x\in X$ is said to be \emph{wild} if every open neighborhood of $x$ contains a loop that is essential, i.e. not null-homotopic, in $X$. The subspace of all wild points of $X$ is denoted $\fw(X)$. The wild set $\fw(X)$ may have its own wild points, so we may consider the descending sequence of closed 
subspaces
$X=\fw^0(X)\supseteq \fw^1(X)\supseteq \fw^2(X)\supseteq\cdots$, where $\fw^{k+1}(X):=\fw(\fw^k(X))$.
We will say that $X$ is \emph{$\fw$-stable} if $\fw^k(X)$ is locally path-connected for all $k$.
Furthermore, we define the \emph{wildness rank} $\rk(X)$ of $X$ as the smallest integer $n$ such 
that $\fw^n(X)=\emptyset$ (and $\rk(X)=\infty$ if there is no such $n$). 

The following are our principal results about the Lusternik-Schnirelmann category $\cat(X)$ and the topological complexity $\tc(X)$ of 1-dimensional Peano continua. Note that a path-connected one-dimensional space is simply-connected if and only if, it does not contain a simple closed curve. 

\begin{theorem}\label{mainthm2}
Assume $X$ is a $\fw$-stable one-dimensional Peano continuum and $\rk(X)=n<\infty$. Then
\[\cat(X)=\begin{cases} n-1 & \text{ if }\fw^{n-1}(X)\text{ contains no simple closed curve,}\\
n & \text{ otherwise}.
\end{cases}
\]
\end{theorem}

\begin{theorem}\label{mainthm4}
Assume $X$ is a $\fw$-stable one-dimensional Peano continuum and $\rk(X)=n<\infty$. Then
\[\tc(X)=\begin{cases} 2 n-2 & \text{ if }\fw^{n-1}(X)\text{ contains no simple closed curve,}\\
 2 n-1\text{ or }2n &  \text{ if } \fw^{n-1}(X)\text{ contains one simple closed curve,} \\
 2 n &  \text{ otherwise.}
\end{cases}
\]
\end{theorem}

For comparison, if $X$ is a one-dimensional CW-complex, then $\cat(X)\in\{0,1\}$
and $\tc(X)\in\{0,1,2\}$. The authors expect that the middle case in Theorem \ref{mainthm4} is always $\tc(X)=2n-1$ whenever $\fw^{n-1}(X)$ contains exactly one simple closed curve. We give an ad-hoc computation in Example \ref{wildcircleexample} as first evidence of this conjecture. However, this improvement remains unproved and will likely require new technical methods.

The organization of the paper is as follows. In the next section we introduce the relevant definitions 
of the Lusternik-Schnirelmann category and topological complexity and establish their properties. 
In Section 3, we discuss some special properties of one-dimensional spaces, define the wild set, and describe some standard decompositions. In Section 4, we study iterated wild sets, define the wildness rank, 
and use it to prove Theorem \ref{mainthm2}. Finally, in Section 5, we give estimates for topological 
complexity of one-dimensional spaces and prove Theorem \ref{mainthm4}.

%=====================================================================================================
\section{Lusternik-Schnirelmann category and topological complexity}\label{secLScatandTC}
%=====================================================================================================

Standard references for LS-category and topological complexity are respectively 
\cite{CLOT} and \cite{Farber2008}. We will use freely the results from those monographs and we will try to use compatible notation whenever possible.

Throughout, all spaces will be assumed to be Hausdorff. Given a space $X$ we will denote by 
$PX$ the space of paths $\ui\to X$, endowed with the compact-open topology. The evaluation maps 
at the points 0,1 or both 0 and 1 are Hurewicz fibrations and are denoted respectively 
$\ev_0,\ev_1\colon PX\to X$ and $\ev_{0,1}\colon PX\to X\times X$.
It is well-known that $\ev_0$ is a homotopy equivalence with a homotopy inverse given by the map 
$c\colon X\to PX$ that to each $x\in X$ assigns the constant path $c_x$. The fibre of $\ev_0$ over a point 
$x\in X$ consists of all paths in $X$ that begin at $x$ and will be denoted $P_\ast X$. The paths
in $P_\ast X$ can be continuously deformed to their initial point, therefore $P_\ast X$ is contractible. 

Our definitions of LS-category and topological complexity are based on closed filtrations and 
are stated as follows. 
A subspace $A\subseteq X$ is said to be \textit{categorical in $X$} if the inclusion map $A\to X$ is 
homotopic to a constant map or equivalently, if the projection $\ev_1\colon P_\ast X\to X$
admits a section over $A$. By default, the empty set is categorical in $X$.  

\begin{definition}\label{lscatdef}
Let $X$ be a path-connected space. The \textit{(normalized) Lusternik-Schnirelmann category} 
$\cat(X)$ of $X$ is the smallest integer $n$ for which there exists a filtration 
$$\emptyset=F_{-1}\subset F_0\subset F_1\subset F_2\subset \cdots \subset F_n=X$$ 
by closed subsets of $X$ such that the difference $F_{j}\backslash F_{j-1}$ is categorical in $X$ 
for all $j\in\{0,1,2,\dots, n\}$. If no such $n$ exists, then $\cat(X)=\infty$.
\end{definition}

A subspace $A\subseteq X\times X$ is said to \textit{admit a motion plan} in $X$ if the 
projections $p_1,p_2\colon A\to X$ are homotopic, or equivalently, if the projection 
$\ev_{0,1}\colon PX\to X\times X$ admits a section over $A$. 
By default, the empty set admits a motion plan.  

\begin{definition}\label{tcdef}
Let $X$ be a path-connected space. The \textit{(normalized) topological complexity} $\tc(X)$ of $X$ is the smallest integer
$n$ for which there exists a filtration 
$$\emptyset=F_{-1}\subset F_0\subset F_1\subset F_2\subset \cdots \subset F_n=X\times X$$
by closed subsets of $X\times X$ such that $F_{j}\backslash F_{j-1}$ admits a motion plan in $X$ for all
$j\in\{0,1,2,\dots, n\}$. If no such $n$ exists, then $\tc(X)=\infty$.
\end{definition}

The above definitions for $\cat(X)$ and $\tc(X)$ in terms of closed filtrations is somewhat different than the standard definition given in terms of suitable open covers (cf. \cite{CLOT} and \cite{Farber2008}) but the idea of using finite filtrations to define these invariants is not new. For path-connected, paracompact spaces $X$, our definition of $\cat(X)$ is analogous to the definition of $\cat(X)$ given by I.M James \cite{James} in terms of finite filtrations of \textit{open} subsets. For Euclidean Neighborhood Retract (ENR) spaces $X$, it is shown in \cite[Proposition 4.12]{Farber2008} that our definition of $\tc(X)$ in terms of closed filtrations agrees with the standard one. If $X$ is a CW-complex, or more generally an Absolute Neighborhood Retract (ANR), the closed filtration definition of both $\cat(X)$ and $\tc(X)$ agrees with the standard one \cite{GarciaCalcines,Srinivasan2}. However, for more general spaces, the values may differ.

Our choice of definition using filtrations by closed subsets is motivated by the fact that any open set in $X$ that contains a wild point of $x$ cannot be categorical in $X$. Hence, the classical definitions of $\cat(X)$ and $\tc(X)$ give infinite value when $X$ contains a wild point. A similar difficulty appears in the study of the topological complexity relative to a map \cite{Pavesic} and in geodesic motion planning \cite{Recio-Mitter}. In both cases, a way out is to base the definitions on closed filtrations. 

Since our definitions are based on closed filtrations we must carefully verify which properties of $\cat(X)$ and $\tc(X)$ that appear in the literature are still valid. We take care of this in the following theorems. 

\begin{theorem}\label{thmcatproperties}
The following statements are valid for arbitrary path-connected spaces $X$ and $Y$.
\begin{enumerate}
\item
$\cat(X)=0$ if and only if $X$ is contractible.
\item
If $X\simeq Y$, then $\cat(X)=\cat(Y)$.
\item
If $X$ contains a closed path-connected subspace $A$ such that $X\backslash A$ can be deformed to 
a discrete subspace of $X$, then
$\cat(X)\le\cat(A)+1$.
\item
$\cat(X\times Y)\le \cat(X)+\cat(Y)$. 
\item 
If $X$ is a locally compact subspace of some Euclidean Neighbourhood Retract, then \
$\cat(X)\ge \text{cup-length} \big(\Ker \check H^\ast(X)\to \check H^\ast(\ast)\big)$.
\end{enumerate}
\end{theorem}

In the last statement, $\check H^\ast$ denotes the \v Cech-cohomology functor with integer coefficients and the cup-length of an ideal $I$ in a cohomology ring is, by definition, the maximal number of factors in a non-zero product of elements of $I$. It is well-known that \v Cech-cohomology agrees with standard singular homology on paracompact, locally contractible spaces (in particular, on CW-complexes) \cite{Spanier66}.
 
\begin{theorem}\label{thmtcproperties}
The following statements are valid for arbitrary path-connected spaces $X$ and $Y$.
\begin{enumerate}
\item
$\tc(X)=0$ if and only if $X$ is contractible.
\item
If $X\simeq Y$, then $\tc(X)=\tc(Y)$.
\item
$\cat(X)\le\tc(X)\le \cat(X\times X)$.
\item
$\tc(X\times Y)\le \tc(X)+\tc(Y)$. 
\item 
If $X$ is a locally compact subspace of some Euclidean Neighbourhood Retract, then \
$\tc(X)\ge \text{cup-length}\big(\Ker \Delta^\ast\colon\check H^\ast(X\times X)\to \check H^\ast(X)\big)$.
\end{enumerate}
\end{theorem}

We will prove both sets of statements together. 

\begin{proof}
Statement \ref{thmcatproperties}(1) follows directly from the definition and 
statement \ref{thmtcproperties}(1) is proved as Theorem 1 in \cite{Farber1}. 

The standard proofs (see for example \cite[Theorem 1.30]{CLOT} and \cite[Theorem 3]{Farber1}) 
of homotopy invariance statements \ref{thmcatproperties}(2) and \ref{thmtcproperties}(2) 
are easily adapted to closed filtrations so we leave the details to the readers.

Toward the proof of \ref{thmcatproperties}(3) observe that a discrete subspace $D$ of a path-connected 
space $X$ is categorical in $X$. In fact, the paths connecting points of $D$ to a point $x_0\in X$ 
define a continuous deformation of $D$ to $x_0$. Clearly, if $X\backslash A$ can be deformed to a discrete subspace 
of $X$, then $X\backslash A$ is categorical in $X$. Therefore, by adding $X \backslash A$ to $A$  we may increase $\cat(A)$ 
at most by one. Note that \ref{thmcatproperties}(3) generalizes the classical estimate of LS-category 
of a mapping cone, cf. \cite[Theorem 1.32]{CLOT}

Estimate \ref{thmtcproperties}(3) can be proved along the same lines  as in the usual proof for open coverings 
- see for example \cite[Theorem 5]{Farber1}.

For proofs of product inequalities \ref{thmcatproperties}(4) and \ref{thmtcproperties}(4), suppose that $\cat(X)=n$ and $\cat(Y)=m$. Then there exist closed filtrations 
$\emptyset\subset F_0\subset F_1\subset \cdots \subset F_n=X$ and 
$\emptyset\subset G_0\subset G_1\subset \cdots \subset G_m=Y$, such that the
differences $F_i\backslash F_{i-1}$ and $G_j \backslash G_{j-1}$ are categorical in $X$ and $Y$, respectively. For 
$k=0,1,\ldots,m+n$ define 
$$H_k:=\bigcup_{i+j=k} F_i\times G_j$$
yielding a closed filtration $H_0\subset H_1\subset \cdots \subset H_{m+n}$ of $X\times Y$.
The difference of successive terms in the filtration is given as
$$H_k\backslash H_{k-1}=\bigcup_{i+j=k} (F_i\backslash F_{i-1})\times (G_j\backslash G_{j-1}).$$
Clearly, each product $(F_i\backslash F_{i-1})\times (G_j\backslash G_{j-1})$ is categorical in $X\times Y$. Moreover, the
terms in the union are separated (the closure of each of them does not intersect the others), so the 
deformation of the individual terms to a point combine into a continuous deformation of 
$H_k\backslash H_{k-1}$ to a point. We may thus conclude that $\cat(X\times Y)\le m+n$ as claimed. The proof of 
the product inequality for topological complexity is analogous (cf. also \cite[(3.8)]{Pavesic2013}). 

The cohomological lower bound for topological complexity defined using closed filtrations (stated as \ref{thmtcproperties}(5)) was proved in detail in \cite[Section 3.5]{Pavesic1}. Note that the relevant proof in \cite{Pavesic1} was added after the publication of 
the homonymous article \cite{Pavesic} and that the results apply to our setting since the topological complexity $\tc(X)$ of a space $X$ agrees with the topological complexity $\tc(id_{X})$ of the identity map on $X$. The analogous statement for LS-category is proved analogously, e.g. as a direct consequence of \cite[Corollary 3.20]{Pavesic1}.
\end{proof}

\begin{example}\label{graphexample}
Let $G$ be a connected 1-dimensional CW-complex (i.e. a connected multigraph). 
Since $G$ can be obtained by adding a discrete set of open segments to its maximal tree, 
Theorem \ref{thmcatproperties}(3) implies that $\cat(G)\le 1$. 
By taking into account Theorem \ref{thmcatproperties}(1) and \ref{thmcatproperties}(2) as well, we obtain
$$\cat(G)=\begin{cases} 0, & \text{ if } G \text{ is a tree,}\\
1, & \text{ otherwise}.
\end{cases}$$

By combining the above result with estimates Theorem \ref{thmtcproperties}(3) and 
Theorem \ref{thmcatproperties}(4) we obtain $\tc(G)\le \cat(G\times G)\le 2\,\cat(G)\le 2$. 
By Theorem \ref{thmtcproperties}(1), $\tc(G)=0$ if and only if $G$ is a tree. 
Let us now consider the circle $S^1$ viewed as the subspace of the complex plane. As in \cite[Example 4.8]{Farber1}, we may filtrate 
$S^1\times S^1$ by $F_0:=\{(x,-x)\mid x\in S^1\}$, $F_1:=S^1\times S^1$, and defining the motion plans (i.e. sections of the projection 
$\ev_{0,1}\colon PS^1\to S^1\times S^1$) explicitly. As a consequence, if $G$ is a graph with exactly one cycle, then $G\simeq S^1$ and so
$\tc(G)=\tc(S^1)=1$. If $G$ has more than one cycle, then one can show that the 
cup-length estimate of \ref{thmtcproperties}(5) gives $\tc(G)\ge 2$, see \cite[Proposition 4.42]{Farber2008}
for details of the computation. Summarizing,  
$$\tc(G)=\begin{cases} 0 & \text{ if } G \text{ is a tree}\\
 1 &  \text{ if } G \text{ has exactly one cycle}, \\
 2 &  \text{ otherwise.}
\end{cases}
$$
\end{example}

We conclude this section with relative versions of LS-category and topological complexity, which are particularly useful when dealing with disconnected subspaces. The relative version of LS-category was first defined by E.R. Fadell (See \cite[Definition 7.1]{CLOT}) and the relative version of topological complexity was introduced by Farber in \cite[Definition 4.20]{Farber2008}. Consistent with our earlier definitions, we choose to define these values using finite filtrations by closed sets. 

\begin{definition}
Let $A$ be a subspace of a path-connected space $X$. The \textit{(normalized) relative  Lusternik-Schnirelmann 
category of $A$ in 
$X$}, denoted $\cat_{X}(A)$,  is the smallest integer $n$ for which there exists a filtration 
$$\emptyset=F_{-1}\subseteq F_0\subseteq F_1\subseteq \cdots \subseteq F_{n}=A$$ 
by closed subsets of $A$ such that $F_{j}\backslash F_{j-1}$ is categorical 
in $X$ for all $j\in\{0,1,2,\dots, n\}$.  If no such $n$ exists, then $\cat_{X}(A)=\infty$.

Similarly, given a subspace $A\subseteq X\times X$ we define the \textit{(normalized) relative  topological complexity of $A$ in $X$}, denoted $\tc_{X}(A)$, as the smallest integer $n$ for which there exists 
a filtration 
$$\emptyset=F_{-1}\subseteq F_0\subseteq F_1\subseteq \cdots \subseteq F_{n}=A$$ 
by closed subsets of $A$ such that the set $F_{j}\backslash F_{j-1}$ admits a motion plan in $X$ for all 
$j\in\{0,1,2,\dots, n\}$. If no such $n$ exists, then $\tc_{X}(A)=\infty$.
\end{definition}

Observe that $\cat_X(X)=\cat X$ and $\tc_X(X\times X)=\tc(X)$. We will later need the following 
estimates of LS-category in terms of relative category.

\begin{proposition}\label{splitprop}
For any closed subspace $A\subseteq X$, we have $\cat(X)\leq \cat_{X}(A)+\cat_{X}(X\backslash A)+1$.
\end{proposition}

\begin{proof}
Assume that $A$ is closed with $\cat_X(A)=n$ and $\cat_X(X\backslash A)=m$, so that there exist closed filtrations 
$\emptyset=F_{-1}\subseteq F_0\subseteq F_1\subseteq \cdots \subseteq F_{n}=A,$ and $G_0\subseteq G_1\subseteq \cdots \subseteq G_{m}=X\backslash A,$
with all differences $F_{i}\backslash F_{i-1}$ and $G_{j}\backslash G_{j-1}$ categorical in $X$. 
Note that the sets $F_i$ are closed in $X$ as well, while the sets $G_j$ may not be closed in $X$ in general.
Set $F_{n+j+1}:=A\cup G_j$ for $0\leq j\leq m$ and observe that $F_{j}$ are closed in $X$ for all $0\leq j\leq m+n$ and that $F_{j}\backslash F_{j-1}$ are 
categorical in $X$ for all $0\leq j\leq m+n+1$. Thus $\cat(X)\leq m+n+1$.
\end{proof}

\begin{proposition}\label{finiteunionprop}
If $X$ is path-connected and $A\subseteq X$ is a finite disjoint union of closed path-connected subspaces $A_1,A_2,\dots,A_k$, then $\cat_X(A)=\sup\{\cat_{X}(A_i)\mid 1\leq i\leq k\}$.
\end{proposition}

\begin{proof}
That $\cat_X(A_i)\leq \cat_{X}(A)$ for each $i$ is clear. If $\sup\{\cat(A_i)\mid 1\leq i\leq k\}=n<\infty$, then for each $i$, there is a closed filtration $F_{i,0}\subseteq F_{i,1}\subseteq \cdots F_{i,n}=A_i$ such that $F_{i,j}\backslash F_{i,j-1}$ is categorical in $X$ for all $0\leq j\leq n$. Then $G_j=\bigcup_{i=1}^{k}F_{i,j}$, $0\leq j\leq n$ defines a closed filtration of $A$. Each $G_{j}\backslash G_{j-1}$ is categorical in $X$ since it is a finite disjoint union of sets which are categorical in $X$. Hence, $\cat_X(A)\leq n$. 
\end{proof}

The next proposition is analogous to that of Theorem \ref{thmcatproperties}(2). We refer to \cite[Proposition 7.2]{CLOT} for the proof.

\begin{proposition}\label{homcatcorpairs}
If $(X,A)$ and $(Y,B)$ are homotopy equivalent pairs of spaces, then $\cat_X(A)=\cat_Y(B)$.
\end{proposition}

%=====================================================================================================
\section{One-dimensional spaces and wild sets}\label{sec:One-dimensional spaces and wild sets}
%=====================================================================================================

The ``dimension" of a space will refer to the Lebesgue covering dimension. A \textit{Peano continuum} is a connected, locally path-connected compact metrizable space. We refer to \cite[Chapter VIII]{Nadler} for the general theory of Peano continua. It follows from the Hahn-Mazurkiewicz Theorem \cite[8.14]{Nadler} that the continuous image of a Peano continuum in a Hausdorff space is a Peano continuum.

A \textit{simple closed curve} is a space that is homeomorphic to $S^1$. It is well-known that a path-connected Hausdorff space $X$ is uniquely arcwise connected if and only if $X$ does not contain a simple closed curve. A Peano continuum $D$ is a \textit{dendrite} if it contains no simple closed curve (equivalently, if it is uniquely arcwise connected). It is well-known that dendrites are one-dimensional and contractible \cite[Chapter X]{Nadler}. 

Since one-dimensional spaces will be our primary focus, we recall some important facts about them. 
First, we note that one-dimensional Hausdorff spaces are aspherical \cite{CurtisFortaspehrical} and have contractible generalized universal coverings in the sense of Fischer-Zastrow \cite[Example 4.14]{FZ07}. 
The following lemma plays a key role in one-dimensional homotopy theory. It was first proved in \cite{Fort} 
and a nice proof may also be found in \cite[Theorem 3.7]{CConedim}. Recall that a loop $S^1\to X$ is said to be \textit{inessential} if it is null-homotopic and \textit{essential} otherwise.

\begin{lemma}\label{dendritelemma}
A loop $\alpha\colon S^1\to X$ in a one-dimensional Hausdorff space $X$ is inessential if and only if there exists a dendrite $D$, a surjective loop 
$\beta\colon S^1\to D$ and a map $f\colon D\to X$ such that $\alpha=f\circ\beta$.
\end{lemma}

It follows that every inessential loop in a one-dimensional Hausdorff space contracts in its own image. 

\begin{corollary}\label{pioneinjectivecor}\cite[Corollary 3.3]{CConedim}
If $X$ is a one-dimensional Hausdorff space and $A\subseteq X$, then the inclusion of $A$ in $X$ is $\pi_1$-injective.
\end{corollary}

\begin{proof}
Let $\alpha:S^1\to A$ be a loop, which is inessential in $X$. By Lemma \ref{dendritelemma}, there is a dendrite $D$, a surjective loop $\beta\colon S^1\to D$, and a map $f\colon D\to X$ such that $\alpha=f\circ\beta$. Since dendrites are contractible, $\alpha$ contracts in $\im(\alpha)=\im(f)$. Thus $\alpha$ is inessential in $A$.
\end{proof}

\begin{corollary}\label{uaccor}
A path-connected one-dimensional Hausdorff space is uniquely arcwise connected if and only if it is simply connected.
\end{corollary}

\begin{proof}
Let $X$ be a path-connected one-dimensional Hausdorff space. Suppose $X$ is uniquely arcwise connected and let $\alpha:S^1\to X$ be a loop. Then $\im(\alpha)$ is a dendrite. Since dendrites are contractible, it follows that $\alpha$ is inessential. Thus $X$ is simply connected. For the converse, suppose $X$ is not uniquely arcwise connected. It is straightforward to show that there exists a simple closed curve $A\subseteq X$. By Corollary \ref{pioneinjectivecor}, the inclusion map $A\to X$ is $\pi_1$-injective. Thus $X$ is not simply connected.
\end{proof}

In the case of Peano continua, we have the following strengthening of Corollary \ref{pioneinjectivecor}.

\begin{lemma}\cite[Lemma 3.1]{CConedim}\label{retractlemma}
Let $A\subseteq X$ be one-dimensional Peano continua. Then $A$ is a retract of $X$.
\end{lemma}

\begin{definition}\label{defwildness}
Let $X$ be any space. A point $x\in X$ is a \textit{wild point of} $X$ if for every neighborhood $U$ of $x$, there exists a based 
loop $\alpha\colon S^1\to U$, which is essential in $X$. We let $\fw(X)$ denote the subspace of $X$ consisting of all wild points of $X$.
\end{definition}

\begin{figure}[ht]
    \centering
    \includegraphics[scale=0.3]{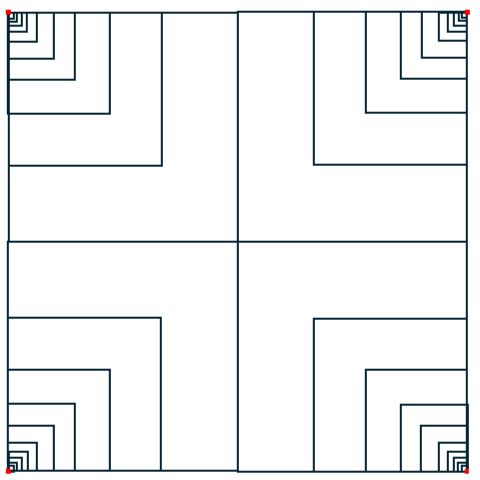}  \hspace{15mm}   \includegraphics[scale=0.21]{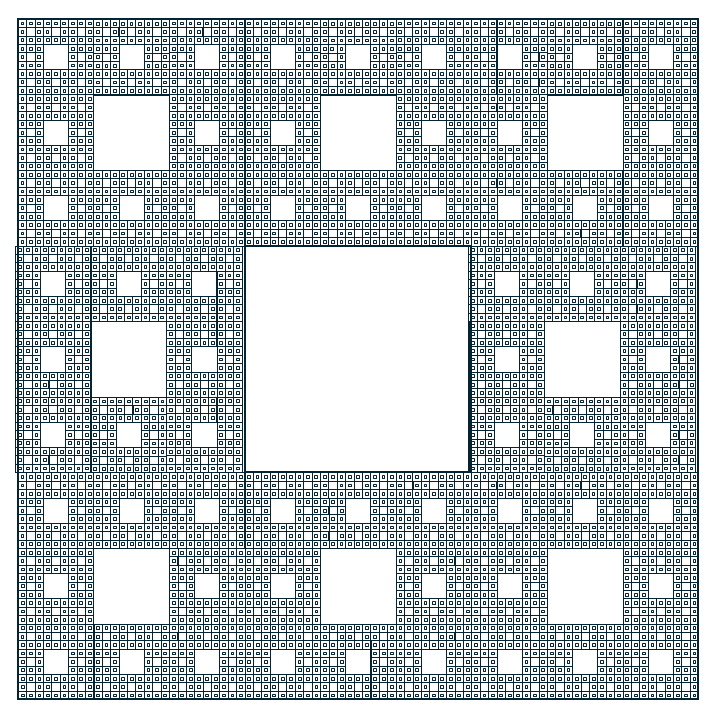}
    \caption{A space whose four corners are wild points (left) and the Sierpinski Carpet (right), 
    in which all points are wild .}
    \label{fig:wild set}
\end{figure} 

\begin{remark}
It is clear from the definition of $\fw(X)$ that $X\backslash\fw(X)$ is open in $X$ and so $\fw(X)$ is closed. 
In particular, if $X$ is a compact metric space, then so is $\fw(X)$. 
If $X$ is locally path-connected, then $\fw(X)$ is precisely the subspace of $X$ consisting of points at which $X$ 
fails to be semi-locally simply connected (the latter being the usual definition of the ``wild set" or ``bad set" 
in the literature on one-dimensional Peano continua).
\end{remark}

The following basic observations about wild sets are analogues of the results about (sequential) based-wild sets in \cite{BrazasMitra}. 
Since the arguments are nearly identical to those in \cite{BrazasMitra} we just sketch the proofs. Recall that a map $f\colon X\to Y$ 
is \emph{$\pi_1$-injective} if for every $x\in X$, the induced homomorphism $f_{\#}\colon\pi_1(X,x)\to \pi_1(Y,f(x))$ is injective.

\begin{lemma}\label{injectivelemma}
If $f\colon X\to Y$ is a $\pi_1$-injective map, then $f(\fw(X))\subseteq \fw(Y)$.
\end{lemma}

\begin{proof}
Let $x\in\fw(X)$ and $U$ be a neighborhood of $f(x)$. Since $x\in\fw(X)\cap f^{-1}(U)$, there exists a loop $\alpha:S^1\to f^{-1}(U)$, which is essential in $X$. Since $f$ is $\pi_1$-injective, $f\circ \alpha:S^1\to U$ is essential in $Y$. Thus $f(x)\in \fw(Y)$.
\end{proof}

The next lemma is a direct combination of Corollary \ref{pioneinjectivecor} and \ref{injectivelemma}.

\begin{lemma}\label{onedimwildinclusion}
If $X$ is a one-dimensional Hausdorff space and $A\subseteq X$, then $\fw(A)\subseteq \fw(X)$.
\end{lemma}

When wild points are absent, a ``deforestation" process developed in \cite{ConnerMeilstrup} implies the following.

\begin{theorem}\label{graphtheorem}
If $X$ is a non-simply connected one-dimensional Peano continuum with $\fw(X)=\emptyset$, then there exists a unique finite graph $G\subseteq X$ with no free edges, which is a deformation retract of $X$.
\end{theorem}

For example, if $X$ is a one-dimensional Peano continuum with a cyclic fundamental group, then $X$ deformation retracts onto a unique simple closed curve in $X$.

When wild points are present, they play a particularly important role in the theory of one-dimensional spaces. For instance, if $X$ is a one-dimensional Peano continuum, then $\fw(X)$ is the set of fixed-points by the action of self-homotopy equivalences on $X$ \cite[Theorem 3.3]{ConnerMeilstrup}. Strict rigidity of the wild set under continuous deformation is vital to Eda's homotopy classification of one-dimensional Peano continua \cite{Edaonedim}, which states that one-dimensional Peano continua are homotopy equivalent if and only if they have abstractly isomorphic fundamental groups.

M. Meilstrup \cite{Meilstrup} proved that every one-dimensional Peano continuum $X$ is homotopy equivalent to a one-dimensional Peano 
continuum $Y$, where $\fw(X)=\fw(Y)$ and $Y\backslash \fw(Y)$ is a disjoint union of a sequence of open arcs $O_1,O_2,O_3,\dots$, 
where $\ov{O_n}$ is a closed arc or simple closed curve, and $\ov{O_n}\backslash O_n\subseteq \fw(Y)$. A modified proof appears 
in \cite{Edaonedim}. The homotopy equivalence $q:X\to Y$ is a quotient map, which collapses to a point each member of countable collection of dendrites. Since we are interested in filtrations of spaces to compute topological complexity, we require the actual decomposition of $X$ determined by $q$. In particular, the following can be proved by combining the results in \cite{Meilstrup}. We give a brief outline of the proof and refer to Meilstrup's Thesis for the details.

\begin{theorem}\label{structuretheorem}
Let $X$ be a one-dimensional Peano continuum with $\fw(X)\neq \emptyset$. Then $\fw(X)$ is contained in a closed subset $Y\subseteq X$ such that
\begin{enumerate}
\item $X\backslash Y$ is a disjoint union of open contractible sets;
\item $\fw(X)$ is a strong deformation retract of $Y$;
\item The closure of every connected component of $Y\backslash \fw(X)$ is a dendrite that meets $\fw(X)$ at a single point.
\end{enumerate}
Moreover, the statement holds if $X$ is a finite disjoint union of one-dimensional Peano continua.
\end{theorem}

\begin{proof}[Sketch of Proof]
Following \cite{Meilstrup}, one identifies a ``core" Peano continuum $C\subseteq X$, which is a strong deformation retract of $X$, contains $\fw(X)$, and has 
the property that if $U_1,U_2,U_3,\dots$ are the connected components of $X\backslash C$, then $E_n=\ov{U}_n$ is a dendrite that meets 
$C$ at a single point $e_n$ (this ``deforestation" process also appears in \cite{ConnerMeilstrup}). One then proceeds to identify a 
(possibly finite) sequence of dendrites $D_1,D_2,D_3,\dots \subseteq C$ and points $d_1,d_2,d_3,\dots \in C$ such that 
\begin{enumerate}
\item for each $j\in\bbn$, $D_j$ meets $\fw(X)$ at a single point $d_j$ for each $j$,
\item $(D_j\backslash \{d_j\})\cap (D_i\backslash \{d_i\})=\emptyset$ if $i\neq j$,
\item if $W=\fw(X)\cup \bigcup_{j}D_j$, then $C\backslash W$ is a (possibly finite) disjoint union of open arcs $A_1,A_2,A_3,\dots$, whose closures are closed arcs disjoint from $\fw(X)$.
\end{enumerate}
(In \cite{Edaonedim}, Eda uses brick partitions to construct such a sequence of dendrites $D_j$). Set $Y=W\cup \bigcup\{E_n\mid e_n\in W\}$. Since adjoining a null sequence of dendrites along a sequence of points in a dendrite gives a dendrite, (2) and (3) follow. Note that a component $V$ of $X\backslash Y$ consists of either a set $U_n$ or an open arc $A_m$ potentially with some of the dendrites $E_n$ attached (if $e_n\in A_m$). In particular, $V$ is a connected subset of a dendrite and is therefore contractible.

The statement regarding finite disjoint unions of Peano continua follows by applying the main statement to each connected component.
\end{proof}

\begin{figure}[ht]
    \centering
    \includegraphics[height=1.6in]{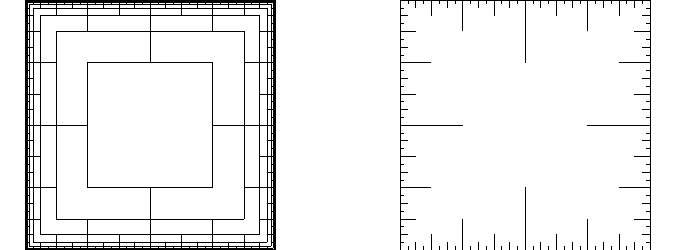}  
    \caption{A Peano Continuum $X$ whose wild set $\fw(X)$ is the outer boundary square (left) and a subspace $Y$ as in Theorem \ref{structuretheorem} that deformation retracts onto $\fw(X)$ and for which $X\backslash Y$ is a disjoint union of open arcs (right). In the more general scenario, $X$ may have a null-sequence of dendrites $E_1,E_2,\dots$ attached to it. Whenever such a dendrite $E_n$ meets $X$ at a point of the subspace $Y$ shown here, we would include it to become part of $Y$. }
    \label{fig:yexample}
\end{figure} 

\begin{corollary}\label{wilditerationcor}
If $X$ is a one-dimensional Peano continuum, then $\cat(X)\leq \cat_{X}(\fw(X))+1$.
\end{corollary}

\begin{proof}
By Theorem \ref{structuretheorem}, there is a closed set $Y\subseteq X$ that strongly deformation retracts onto $\fw(X)$ and such that $X\backslash Y$ is a disjoint union of open contractible sets. Thus $\cat_X(X\backslash Y)=0$. By Proposition \ref{splitprop}, we have $\cat(X)\leq \cat_{X}(Y)+1$. Since the identity map gives a homotopy equivalence of pairs $(X,\fw(X))\to (X,Y)$, we have $\cat_X(Y)=\cat_X(\fw(X))$ by Proposition \ref{homcatcorpairs}.
\end{proof}

The following theorem covers the case where the wild set of a one-dimensional Peano continuum is zero-dimensional, e.g. homeomorphic to the Cantor set (see Figure \ref{fig:cantor}). We note that this case is not covered by our main results because $0$-dimensional continua are not locally path connected unless they are finite and discrete and thus the $\fw$-stable condition may not be met.

\begin{theorem}\label{dendritetheorem}
If $X$ is a one-dimensional Peano continuum with $\fw(X)\neq\emptyset$ and there exists a dendrite $D$ such that $\fw(X)\subseteq D\subseteq X$, then $\cat(X)=1$ and $\tc(X)=2$. In particular, this holds if $\fw(X)$ is zero-dimensional.
\end{theorem}

\begin{proof}
Since $X$ is not contractible $1\leq \cat(X)$. Corollary \ref{wilditerationcor} gives $\cat(X)\leq \cat_{X}(\fw(X))+1$. Since $D$ is contractible and $\fw(X)\subseteq D$, $\fw(X)$ is categorical in $X$. Thus $\cat_{X}(\fw(X))=0$. We conclude that $\cat(X)=1$. 

By the results of Section \ref{secLScatandTC}, we have $\tc(X)\leq \cat(X\times X)\leq 2\cat(X)=2$. We claim that since $\fw(X)\neq\emptyset$, there exists a finite graph $G\subseteq X$ with at least two cycles (in fact, a graph with any finite number of cycles exists). Pick a point $x\in \fw(X)$ and a simply closed curve $C\subseteq X$. Let $U$ be a path-connected neighborhood of $x$ that does not contain all of $C$. Since $x\in \fw(X)$, $U$ cannot be simply connected. By Corollary \ref{uaccor}, there must be a simple closed curve $D\subseteq U$. If $C\cap D=\emptyset$, we can define $G$ to be the union of $C$, $D$, and an arc connecting them. If $|C\cap D|=1$, we take $G=C\cup D$. If $|C\cap D|>1$, we can find an arc $A\subseteq D$ whose endpoints lie in $C$ and whose interior is disjoint from $C$ and we set $G=C\cup A$. Note that for such a graph $G$, we have $\tc(G)=2$ by Example \ref{graphexample}. By Lemma \ref{retractlemma}, $G$ is a retract of $X$, giving $\tc(G)\leq \tc(X)$. Thus $\tc(X)=2$.

For the second statement, we note that Katsuya Eda has proved that if $Y$ is a one-dimensional Peano continuum with $\fw(Y)$ zero-dimensional, then $Y$ is homotopy equivalent to a one-dimensional Peano continuum $X$ and with the property that there exists a dendrite $D\subseteq X$ with $\fw(X)\subseteq D$ \cite[Lemma 2.4]{Edazerodim}. By the homotopy invariance of LS-category and the first part of the theorem, we have $\cat(Y)=\cat(X)=1$ and $\tc(Y)=\tc(X)=2$.
\end{proof}

\begin{figure}[ht]
    \centering
    \includegraphics[height=1.6in]{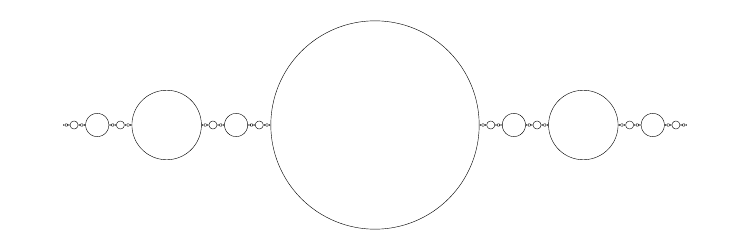}  
    \caption{A Peano Continuum $X$ whose wild set $\fw(X$) is the ternary Cantor set. Theorem \ref{dendritetheorem} applies to $X$ since there is an arc which contains $\fw(X)$. }
    \label{fig:cantor}
\end{figure}

%=====================================================================================================
\section{Iterated Wild Sets and LS-Category}\label{sec:Iterated Wild Sets}
%=====================================================================================================

Since the wild set $\fw(X)$ of a given space $X$ is a space in its own right, it is possible that $\fw(X)$ has some internal wild points, that is, $\fw(\fw(X))$ may be non-empty. Here, we consider the iteration of the wild set construction on spaces.

\begin{definition}[iterated wild sets]
Let $X$ be a space. Set $\fw^{0}(X)=X$. For each $n\in\bbn$, set $\fw^n(X)=\fw(\fw^{n-1}(X))$. We refer to $\fw^{n}(X)$ as the \textit{$n$-th iterated wild set} of $X$.
\end{definition}

Note that $X\supseteq \fw(X)\supseteq \fw^2(X)\supseteq \fw^3(X)\supseteq \cdots$ is a nested sequence of closed subspaces of $X$. It is possible that this sequence never stabilizes or that it stabilizes at a non-empty space. For example, $\fw(S)=S$ if $S$ is the Sierpinski Carpet.

\begin{definition}
The \textit{wildness rank} of a non-empty space $X$ is the smallest positive integer $\rk(X)=n$ such that $\fw^n(X)=\emptyset$. If no such $n$ exists, we say the wildness rank is infinite and write $\rk(X)=\infty$.
\end{definition}

\begin{remark}\label{iterationremark}
It follows from \cite[Theorem 1.2]{BrazasMitra} that for any (one-dimensional) compact metric space $Y$, there exists a (one-dimensional) Peano continuum $X$ such that $\fw(X)=Y$. Hence, if $X$ is a one-dimensional Peano continuum, then $\fw(X)$ may be an arbitrary compact metric space of dimension $\leq 1$. Moreover, iterating the construction $X\mapsto Y$, one may construct (one-dimensional) Peano continua of rank $n$ for any $n\ge 0$. 
\end{remark}

We will use Remark \ref{iterationremark} and the following definition to give an example of a space whose sequence of iterated wild sets fails to stabilize.

\begin{definition}
The \textit{shrinking wedge} of countable set $\{(A_j,a_j)\}_{j\in J}$ of based spaces is the space $\sw_{j\in J}(A_j,a_j)$ whose underlying set is the usual one-point union $\bigvee_{j\in J}(A_j,a_j)$ with canonical basepoint $b_0$ and where $A_j$ is identified canonically as a subset. A set $U$ is open in $\sw_{j\in J}A_j$ if
\begin{itemize}
\item $U\cap A_j$ is open in $A_j$ for all $j\in J$,
\item and whenever $b_0\in U$, we have $A_j\subseteq U$ for all but finitely many $j\in J$.
\end{itemize}
\end{definition}

\begin{example}\label{nonstableexample}
The shrinking wedge of circles $X_2=\sw_{j\in\bbn}(S^1,(1,0))$ is homeomorphic to the usual infinite earring space. If $b_0$ is the canonical wedgepoint of $X_2$, then $\fw(X_2)=\{b_0\}$ and $\fw^2(X_2)=\emptyset$. Thus $\rk(X_2)=2$. As noted in Remark \ref{iterationremark}, we may recursively apply \cite[Theorem 1.2]{BrazasMitra} to find a sequence of one-dimensional Peano continua $X_2\subseteq X_3\subseteq X_4\subseteq \cdots$ so that $\fw(X_{n+1})=X_n$ for all $n\geq 2$. In particular, $X_{n+1}$ can be constructed by attaching a sequence of based circles of null-diameter to $X_n$ by identifying the basepoint of each circle with a point from a fixed countable dense subset of $X_n$. See Figure \ref{fig:rankexample} for planar continua have the homotopy type of $X_3$ and $X_2$ respectively. With the recursive construction completed, we observe that $\rk(X_n)=n$ for all $n\geq 2$. 

Let $Y=\sw_{n\geq 2}(X_n,b_0)$ be the shrinking wedge of a copy of $X_n$ for each $n\geq 2$ and let $y_0$ denote the canonical wedgepoint of $Y$. Then $Y$ is a one-dimensional Peano continuum and there is a shift map $\sigma:Y\to Y$, which maps the $n$-th summand $X_n$ into $(n+1)$-st summand $X_{n+1}$ by the inclusion map. The wild set of $Y$ is $\fw(Y)=\sw_{n\geq 2}\fw(X_n)$ and is equal to the image $\sigma(Y)$ of the shift map. It follows that $\fw(Y)$ is proper subspace of $Y$ that is homeomorphic to $Y$. Iteration gives $Y\supseteq \fw(Y)\supseteq \fw^2(Y)\supseteq \fw^3(Y)\supseteq \cdots$ where $\fw^k(Y)=\sigma^k(Y)$ is a ``smaller" homeomorphic copy of $Y$ for all $k\geq 2$. In particular, the sequence $\{\fw^{k}(Y)\}_{k\in\bbn}$ does not stabilize at any integer $k$ and we have $\bigcap_{k\in\bbn}\fw^k(Y)=\{y_0\}$.
\end{example}

\begin{figure}[ht]
    \centering
    \includegraphics[height=1.6in]{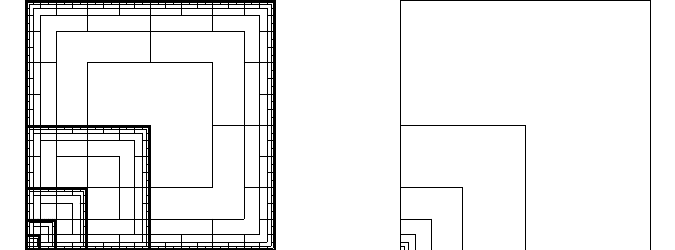}  
    \caption{A one-dimensional Peano continuum $X$ (left) whose wild set $\fw(X)$ (right) is homotopy equivalent to the infinite earring space and thus $\rk(X)=3$.}
    \label{fig:rankexample}
\end{figure} 

\begin{remark}
Our definition of wildness rank mimics the definition of familiar set-theoretic ranks since it behaves like the Cantor-Bendixson rank in many respects. Example \ref{nonstableexample} suggests that the sequence can continue non-trivially into higher infinite ordinals. However, since we show that $\cat(X)$ and $\tc(X)$ are infinite whenever $X$ is a one-dimensional Hausdorff space and $\rk(X)$ is not a finite ordinal (see Corollary \ref{catinfcor} below), we do not continue the definition of iterated wild sets into transfinite ordinals here.
\end{remark}

\begin{definition}
A space $X$ is \textit{$\fw$-stable} if the $n$-th iterated wild set $\fw^n(X)$ is locally path-connected for all $n\in\bbn$.
\end{definition}

\begin{proposition}\label{stableprop}
If $X$ is a $\fw$-stable Peano continuum, then for all $n\in\bbn$, $\fw^n(X)$ is a disjoint union of finitely many Peano continua whenever it is non-empty.
\end{proposition}

\begin{proof}
Since $X$ is $\fw$-stable, each path component of $\fw^n(X)$ is open in $\fw^n(X)$. Moreover, since $X$ is compact, $\fw^n(X)$ is also compact. Thus $\fw^n(X)$ must have finitely many path components. This implies that the path components of $\fw^n(X)$ are also closed in $X$. It follows that each path-component of $\fw^n(X)$ is a Peano continuum.
\end{proof}

The proof of the next lemma is straightforward and is left as an exercise.

\begin{lemma}\label{productlemma}
Let $X$ and $Y$ be spaces. Then $\fw(X\times Y)=X\times\fw(Y)\cup \fw(X)\times Y$.
\end{lemma}

\begin{lemma}\label{iteratedwildinclusionlemma}
Let $X$ be a one-dimensional Hausdorff space. Then for all $n\geq 0$, we have $\bigcup_{i+j=n}\fw^{i}(X)\times \fw^{j}(X)\subseteq \fw^n(X\times X)$.
\end{lemma}

\begin{proof}
Before proving the statement in the lemma, we make a general observation. Fix $i,j$ and $\fw^{i}(X)\times\fw^j(X)\subseteq Z\subseteq X\times X$. It follows from Lemma \ref{onedimwildinclusion} that each inclusion map $\fw^{i}(X) \to X$ is $\pi_1$-injective. Thus the product embedding $\fw^{i}(X)\times\fw^j(X)\to X\times X$ is $\pi_1$-injective for all $i,j$. Moreover, the inclusion $\fw^{i}(X)\times\fw^j(X)\to  Z$ must be $\pi_1$-injective and we have $\fw(\fw^{i}(X)\times\fw^j(X))\subseteq \fw(Z)$ by Lemma \ref{onedimwildinclusion}.

The case of $n=0$ is clear and the case $n=1$ is one inclusion of the equality in Lemma \ref{productlemma}. The cases $n\geq 2$ are proved by induction on $n$. Supposing $\fw^{i}(X)\times \fw^{j}(X)\subseteq \fw^n(X\times X)$, we have
\[\fw^{i}(X)\times \fw^{j+1}(X)\cup\fw^{i+1}(X)\times \fw^{j}(X) = \fw(\fw^{i}(X)\times \fw^{j}(X))\subseteq \fw^{n+1}(X\times X)\]
where the equality is an application of Lemma \ref{productlemma} and the inclusion follows from the observation made in the first paragraph.
\end{proof}

\begin{theorem}\label{hetheorem}
If $f:X\to Y$ and $g:Y\to X$ are inverse homotopy equivalences, then $f|_{\fw(X)}:\fw(X)\to \fw(Y)$ and $g|_{\fw(Y)}:\fw(Y)\to\fw(X)$ are inverse homotopy equivalences. In particular, $(X,\fw(X))$ and $(Y,\fw(Y))$ are homotopy equivalent pairs.
\end{theorem}

\begin{proof}
The restricted functions are well-defined by Lemma \ref{injectivelemma}. If $H:X\times \ui\to X$ is a homotopy from $id_{X}$ to $g\circ f$, then $H$ is $\pi_1$-injective and $H(\fw(X\times\ui))=H(\fw(X)\times\ui)\subseteq \fw(X)$ using Lemma \ref{productlemma} for the first equality. Thus $H$ restricts to a map $\fw(X)\times\ui\to \fw(X)$ that is a homotopy from $id_{\fw(X)}$ to $g|_{\fw(Y)}\circ f|_{\fw(X)}$. The desired homotopy for the other composition can be obtained using the same argument.
\end{proof}

The next corollary follows by combining Theorems \ref{homcatcorpairs} and \ref{hetheorem}.

\begin{corollary}\label{hecatcor}
If $X$ and $Y$ are homotopy equivalent spaces, then $\cat_{X}(\fw(X))=\cat_{Y}(\fw(Y))$.
\end{corollary}

\begin{corollary}
Wild rank is an invariant of homotopy type.
\end{corollary}

\begin{proof}
Suppose $X\simeq Y$. Then Theorem \ref{hetheorem} implies that $\fw^n(X)\simeq \fw^n(Y)$ for all $n\in\bbn$. Since a non-empty space cannot be homotopy equivalent to an empty space, it follows that $\fw^n(X)=\emptyset$ if and only if $\fw^n(Y)=\emptyset$. Thus $\rk(X)=\rk(Y)$.
\end{proof}

\begin{lemma}\label{homlemma}
If $X$ is a one-dimensional Hausdorff space and $H:X\times\ui\to X$ satisfies $H(x,0)=x$ for all $x\in X$, then $H(x,t)=x$ for all $(x,t)\in\fw(X)\times\ui$.
\end{lemma}

\begin{proof}
Note that if $H(x_0,s)\neq x_0$ for some $x_0\in \fw(X)$ and $s>0$, we can define $H':X\times\ui\to X$ by $H'(x,t)=H(x,st)$, which satisfies $H'(x_0,1)=H(x_0,s)=\neq x_0$. Therefore, it suffices to show that $H(x,1)=x$ for all $x\in\fw(X)$. Fix $x\in \fw(X)$ and suppose $H(x,1)=y\neq x$. Define map $f:X\to X$ by $f(x)=H(x,1)$. Let $U$ and $V$ be disjoint open neighborhoods of $x$ and $y$ respectively. Find a neighborhood $W$ of $x$ such that $H(W\times\{0\})\subseteq U$ and $H(W\times\{1\})\subseteq V$. Since $x\in \fw(X)$, we may find an essential loop $\alpha:S^1\to W$. Now $\alpha$ and $f\circ\alpha$ are freely homotopic loops in the one-dimensional Peano continuum $H(\im(\alpha)\times\ui)$ having disjoint images. However, it is well-known that freely homotopic essential loops with disjoint images do not exist in one-dimensional metrizable spaces (see part (b) in the proof of \cite[Theorem 9.13]{BFpants}).
\end{proof}

Applying Lemma \ref{homlemma} to the homotopies in the proof of Theorem \ref{hetheorem}, we obtain the following.

\begin{corollary}\label{hequivonedcor}
If $f:X\to Y$ and $g:Y\to X$ are inverse homotopy equivalences of one-dimensional Hausdorff spaces, then $f|_{\fw(X)}:\fw(X)\to \fw(Y)$ and $g|_{\fw(Y)}:\fw(Y)\to\fw(X)$ are inverse homeomorphisms.
\end{corollary}

\begin{lemma}\label{stablfiltrationlemma}
Suppose $X$ is a $\fw$-stable one-dimensional Peano continuum such that $\rk(X)=n<\infty$. Then there exists closed sets $\emptyset=Y_n\subseteq  Y_{n-1}\subseteq \cdots \subseteq Y_{2}\subseteq Y_{1}\subseteq Y_0 =X$ such that
\begin{enumerate}
\item $\fw^j(X)\subseteq Y_j$ for all $1\leq j\leq n-1$,
\item $Y_j$ strongly deformation retracts onto $\fw^j(X)$ for all $1\leq j\leq n-1$,
%\item If $U$ is a connected component of $Y_j\backslash \fw^j(X)$, then $\ov{U}$ is a dendrite meeting $\fw^j(X)$ at one point,
\item $Y_{j-1}\backslash Y_{j}$ is a disjoint union of contractible sets that are open in $Y_{j-1}$ for all $1\leq j\leq n-1$,
\item $Y_j$ is a finite disjoint union of Peano continua for $1\leq j\leq n-1$,
\item $\fw^{j+1}(X)=\fw(Y_j)$ for all $1\leq j\leq n-1$.
\end{enumerate}
\end{lemma}

\begin{proof}
Since $X$ is $\fw$-stable, it follows from Proposition \ref{stableprop} that $\fw^j(X)$ is a finite disjoint union of Peano continua for all $0\leq j\leq n-1$. Applying Theorem \ref{structuretheorem} to $Y_0=X$, we obtain $Y_1$ satisfying Conditions (1)-(3). Since $\fw(X)$ is a finite disjoint union of Peano continua and $Y_1$ consists of $\fw(X)$ and a null-sequence of attached dendrites, $Y_1$ is a finite disjoint union of Peano continua that deformation retracts onto $\fw(X)$. Thus (4) holds for $j=1$. Since the inclusion $\fw(Y_0)\to Y_1$ is a homotopy equivalence, we have $\fw^2(Y_0)=\fw(Y_1)$ by Corollary \ref{hequivonedcor}. Thus (5) holds for $j=1$.

If $Y_{j}$ satisfies (1)-(5) for given $1\leq j\leq n-2$, we apply the final statement of Theorem \ref{structuretheorem} (for finite disjoint unions of Peano continua) to $Y_{j}$ to obtain closed subspace $Y_{j+1}$ of $Y_j$ such that $\fw(Y_j)\subseteq Y_{j+1}\subseteq Y_j$ where $\fw(Y_j)$ is strong deformation retract of $Y_{j+1}$ and $Y_j\backslash Y_{j+1}$ is a disjoint union of contractible sets that are open in $Y_j$. We have $\fw^{j+1}(X)=\fw(Y_j)$ by the induction hypothesis, immediately implies (1) and (2). Additionally, (3) holds by part (1) of Theorem \ref{structuretheorem}. Recall from the first sentence of the proof that $\fw^{j+1}(X)=\fw(Y_j)$ is a finite disjoint union of Peano continua. Since $\fw(Y_j)$ is a finite disjoint union of Peano continua and $Y_{j+1}$ consists of $\fw(Y_j)$ and a null-sequence of attached dendrites, $Y_{j+1}$ is a finite disjoint union of Peano continua. Thus (4) holds. Finally, since the inclusion $\fw(Y_j)\to Y_{j+1}$ is a homotopy equivalence, Corollary \ref{hequivonedcor} gives $\fw^2(Y_j)=\fw(Y_{j+1})$. Thus $\fw^{j+2}(X)=\fw^2(Y_j)=\fw(Y_{j+1})$, giving (5) and finishing the induction. We remark that the last index in (5) gives $Y_n=\fw(Y_{n-1})=\fw^{n}(X)=\emptyset$ since $\rk(X)=n$.
\end{proof}

\begin{lemma}\label{stableupperboundlemma}
Suppose $X$ is a $\fw$-stable one-dimensional Peano continuum such that $\rk(X)=n<\infty$. Then $\cat(X)\leq n$. Moreover, if $\fw^{n-1}(X)$ contains no simple closed curve, then $\cat(X)\leq n-1$.
\end{lemma}

\begin{proof}
Lemma \ref{stablfiltrationlemma} gives a filtration of closed sets $Y_{n-1}\subseteq \cdots \subseteq Y_2\subseteq  Y_1\subseteq Y_0= X$ each of which is a finite disjoint union of Peano continua, such that each set $Y_{j-1}\backslash Y_j$ is a topological sum of contractible sets, and such that $Y_j$ strongly deformation retracts onto $\fw^j(X)$ for $1\leq j\leq n-1$. 

If $\fw^{n-1}(X)$ contains no simple closed curve, then $Y_{n-1}$ contains no simple closed curve. In this case, it follows that $Y_{n-1}$ is a finite disjoint union of dendrites and is therefore categorical in $X$. Set $Y_n=\emptyset$. The filtration $\{Y_{n-j}\}_{1\leq j\leq n}$ proves that $\cat(X)\leq n-1$.

If $\fw^{n-1}(X)$ contains a simple closed curve, then so does $Y_{n-1}$. In this case, let $P_1,P_2,\dots,P_m$ denote the connected components of $Y_{n-1}$. Since $\fw^{n}(X)=\fw(Y_{n-1})=\emptyset$ and $Y_{n-1}$ deformation retracts onto $\fw^{n-1}(X)$, each $P_i$ is a Peano continuum with no wild points. Theorem \ref{graphtheorem} implies the existence of a finite graph $G_i\subseteq P_i$ that is a strong deformation retract of $P_i$ (take $G_i$ to be a point if $P_i$ is simply connected). Choose a maximal spanning tree $T_i$ in $G_i$ for each $i$ and let $Y_n=\bigcup_{i=1}^{m}T_i$. The set $Y_{n-1}\backslash Y_{n}$ is a topological sum of sets that contain no simple closed curves and is therefore categorical in $X$. The set $Y_n$ is a finite disjoint union of finite trees and is therefore categorical in $X$. Set $Y_{n+1}=\emptyset$. The filtration $\{Y_{n-j}\}_{0\leq j\leq n}$ proves that $\cat(X)\leq n$.
\end{proof}

%\begin{theorem}\label{iteratepctheorem}
%Let $X$ be a one-dimensional Peano continuum and $Y\subseteq X$ be a finite disjoint union of Peano continua. Then $\cat_{X}(Y)\leq\cat_{X}(\fw(Y))+1$.
%\end{theorem}

%\begin{proof}
%We first consider the case where $Y$ is a Peano continuum. In this case $\cat_{X}(Y)\leq \cat(Y)$ holds. By Theorem \ref{standardformtheorem}, there exists a one-dimensional Peano continuum $Y'$ in standard form that is homotopy equivalent to $X$. Then $\cat(Y')=\cat(Y)$ by Theorem \ref{hetheorem} and $\cat_{Y'}(\fw(Y'))=\cat_{Y}(\fw(Y))$ by Corollary \ref{hecatcor}. We have $\cat(Y')\leq \cat_{Y'}(\fw(Y'))+\cat_{Y'}(Y\backslash \fw(Y'))+1$ by Proposition \ref{splitprop}. Since $Y'$ is in standard form, $Y'\backslash \fw(Y')$ is a disjoint union of open contractible sets and thus $\cat_{Y'}(Y'\backslash \fw(Y'))=1$. The theorem follows. Since sets which are categorical in $Y$ are also categorical in $X$, we have $\cat_{Y}(\fw(Y))\leq \cat_{X}(\fw(Y))$. This concludes the proof of the second inequality when $Y$ is a Peano continuum.

%If $Y$ is a disjoint union of Peano continua $Y_1,Y_2,\dots,Y_k$, then $\cat_{X}(Y)=\max\{\cat_{X}(Y_i)\mid 1\leq i\leq k\}$ (for one inequality note that because $Y_i$ is clopen in $Y$, if a set $S\subseteq Y$ is categorical in $X$, then $S\cap Y_i$ is categorical in $X$). Similarly, because $\fw(Y)$ is a disjoint union of clopen sets $\fw(Y_i)$, we have $\cat_{X}(\fw(Y))=\max\{\cat_{X}(\fw(Y_i))\mid 1\leq i\leq k\}$. The general inequality now follows by applying the first paragraph to the Peano continua $Y_i$.
%\end{proof}

Here we use iterated wild sets to obtain lower bounds on the category of a one-dimensional space.

\begin{lemma}\label{steplemma}
Assume $A\subseteq B\subseteq X$ where $X$ is one-dimensional Hausdorff and $B\backslash A$ is categorical in $X$. Then $Y\subseteq B$ implies $\fw(Y)\subseteq A$.
\end{lemma}

\begin{proof}
Suppose $x\in \fw(Y)\backslash A$. Since $Y\backslash A$ is an open neighborhood of $x$ in $Y$, there exists an essential loop $\alpha:S^1\to Y\backslash A$. Since $X$ is $1$-dimensional, the inclusion map $Y\backslash A\to X$ is $\pi_1$-inective (by Lemma \ref{onedimwildinclusion}) and thus $\alpha:S^1\to X$ is essential. Since $B\backslash A$ is assumed to be categorical in $X$, $\alpha$ cannot have image in $B$. Thus there exists a point $y\in \im(\alpha)\backslash B\subseteq Y\backslash B$.
\end{proof}

\begin{lemma}\label{rankboundlemma}
If $X$ is a one-dimensional Hausdorff space and $\cat(X)=m<\infty$, then $\rk(X)\leq m+1$. Moreover, if $\fw^{m}(X)$ has a non-simply connected path-component, then $\rk(X)\leq m$.
\end{lemma}

\begin{proof}
If $\cat(X)=m<\infty$, then there exists closed sets $\emptyset =F_{-1}\subseteq F_0\subseteq F_1\subseteq F_2\subseteq \cdots \subseteq F_{m}=X$ such that $F_{j}\backslash F_{j-1}$ is categorical in $X$ for $0\leq j\leq m$. Applying Lemma \ref{steplemma} in the case where $Y=X$, $A=F_{m-1}$, $B=F_{m}$, we have $\fw(X)\subseteq F_{m-1}$. Applying Lemma \ref{steplemma} in the case where $Y=\fw(X)$, $A=F_{m-2}$, and $B=F_{m-1}$, we have $\fw^{2}(X)\subseteq F_{m-2}$. Proceeding recursively, we eventually obtain $\fw^{j}(X)\subseteq F_{m-j}$ for $0\leq j\leq m+1$. In particular, $\fw^{m+1}(X)\subseteq F_{-1}=\emptyset$, which gives $\rk(X)\leq m+1$.

For the second statement, suppose $\fw^m(X)$ has a non-simply connected path component and, to obtain a contradiction, that $\rk(X)>m$. The recursion from the first paragraph gives $\emptyset=F_{-1}\neq \fw^{m}(X)\subseteq F_0$ where $F_0$ is categorical in $X$. By assumption, $\fw^m(X)$ admits loops that are essential. Applying Lemma \ref{onedimwildinclusion} to the inclusions $\fw^m(X)\to F_0\to X$, we see that $F_0$ admits loops that are essential in $X$. However, this contradicts the fact that $F_0$ is categorical in $X$. 
\end{proof}

\begin{corollary}\label{catinfcor}
Let $X$ be a one-dimensional Hausdorff space. If $\rk(X)=\infty$, then $\cat(X)=\infty$ and $\tc(X)=\infty$.
\end{corollary}

\begin{proof}
Given $\rk(X)=\infty$, $\cat(X)=\infty$ follows from the previous lemma. Since $\cat(X)\leq \tc(X)$ (recall Theorem \ref{thmtcproperties}), we also have $\tc(X)=\infty$.
\end{proof}

\begin{example}
If $X$ is a one-dimensional Hausdorff space and there exists a non-empty subspace $A\subseteq X$ such that $\fw(A)=A$, e.g. if it contains a copy of the Sierpinski Carpet, then no finite motion plan exists for $X$ even with our modified definitions. Indeed, for such a space, Lemma \ref{onedimwildinclusion} implies that $A= \fw^n(A)\subseteq \fw^n(X)$ for all $n\geq 0$ and thus $\rk(X)=\infty$. Corollary \ref{catinfcor} then gives $\cat(X)=\infty$ and $\tc(X)=\infty$.
\end{example}

At this point, we obtain exact values for LS-category in the $\fw$-stable case.

\begin{proof}[Proof of Theorem \ref{mainthm2}]
Suppose $X$ is a $\fw$-stable one-dimensional Hausdorff space with $\rk(X)=n<\infty$. The upper bound $\cat(X)\leq n$ (and $\cat(X)\leq n-1$ if $\fw^{n-1}(X)$ contains no simple closed curve) is given in Lemma \ref{stableupperboundlemma}. The lower bound $n-1\leq \cat(X)$ (and $n\leq \cat(X)$ if $\fw^{n-1}(X)$ contains a simple closed curve) is given in Lemma \ref{rankboundlemma}.
\end{proof}

\section{Topological complexity of one-dimensional spaces}

In this section we prove Theorem \ref{mainthm4}, which generalizes neatly the computation of the topological 
complexity of graphs described in Example \ref{graphexample}.

\begin{proposition}\label{a2prop}
Let $X$ be a one-dimensional Peano continuum and $G\subseteq X$ be a (not necessarily connected) finite graph with more than one cycle. If $\emptyset \subseteq A_0\subseteq A_1\subseteq X\times X$ are closed sets and $G\times G\subseteq A_1$, then at least one of $A_0$ or $A_1\backslash A_0$ does not admit a motion plan in $X$.
\end{proposition}

\begin{proof}
Suppose to obtain a contradiction that $A_0$ and $A_1\backslash A_0$ admit a motion plan in $X$. If $G$ has a connected component $H$, which has more than one cycle, then $H\cap A_0$ and $H\cap (A_1\backslash A_0)$ admit a motion plan in $X$. However, $H$ is a retract of $X$ by Lemma \ref{retractlemma}. Thus $H\cap A_0$ and $H\cap (A_1\backslash A_0)$ admit a motion plan in $H$. However, this contradicts the computation in 
Example \ref{graphexample}.

We now consider the case where each connected component of $G$ has either no cycles or one cycle. Let $C_1$ and $C_2$ be two 
disjoint components of $G$ each containing one cycle. Set $C=C_1\cup C_2$, $K_0=(C\times C)\cap A_0$, and
$K_1=(C\times C)\cap A_1=C\times C$. Then $K_0$ and $K_1\backslash K_0$ admit motion plans in $X$. Since $X$ is Hausdorff and path connected, there exists an injective path 
$\alpha:\ui\to X$ with $\alpha(0)\in C_1$ and $\alpha(1)\in C_2$. Since $C_1$ and $C_2$ are closed disjoint sets, there must exist some connected component $(s,t)$ of $\alpha^{-1}(\im(\alpha)\backslash (C_1\cup C_2))$ such that $\alpha(s)\in C_1$ and $\alpha(t)\in C_2$. In particular, $A=\alpha([s,t])$ is an arc that meets $C_1$ at one endpoint, meets $C_2$ at the other endpoint, and whose interior is disjoint from $C_1\cup C_2$. Note that $\widetilde C=C_1\cup A\cup C_2$ is a connected graph with exactly two cycles. 
Since $\widetilde C$ is a retract of $X$ by Lemma \ref{retractlemma}, $K_0$ and $K_1\backslash K_0$ 
admit motion plans in $\widetilde C$. In other terms, the relative topological complexity 
$\tc_{\widetilde C}(C\times C)$ is at most 1. However, we will show that by a cup-length argument 
$\tc_{\widetilde C}(C\times C)\ge 2$. 

First observe, that $\tc_{\widetilde C}(C\times C)$ can be equivalently described as the topological
complexity of the inclusion map $i\colon C\hookrightarrow \widetilde C$ in the sense of 
\cite[Example 3.3]{Scott}. 
Then \cite[Theorem 3.11]{Scott} provides a cohomological estimate for relative topological complexity 
as $\tc_{\widetilde C}(C\times C)\ge \text{ cup-length }\big((i\times i)^*(\Ker \Delta^*)\big),$
where $\Delta^*\colon H^*(\widetilde C\times \widetilde C)\to H^*(\widetilde C)$ is induced by the
diagonal map $\Delta\colon\widetilde C\to\widetilde C\times \widetilde C$ and 
$(i\times i)^*\colon H^*(\widetilde C\times \widetilde C)\to H^*(C\times C)$ is induced by the inclusion
$i\times i\colon C\times C\hookrightarrow \widetilde C\times \widetilde C$. 
(Note that $C$ and $\widetilde C$ are ENRs, so we may use singular cohomology in our computations.) 
Since the homomorphism 
$(i\times i)^*$ is clearly injective, it follows that 
$\text{cup-length }\big((i\times i)^*(\Ker \Delta^*)\big)\ge \text{ cup-length }(\Ker \Delta^*).$
It is well-known (see for example \cite[Prop. 4.42]{Farber2008}) that 
$\text{cup-length }(\Ker \Delta^*)=2$, which concludes our proof.
\end{proof}

\begin{lemma}\label{torusobstructionlemma}
Suppose that $A\subseteq X\times X$. If there exists an essential loop $\alpha:S^1\to X$ and 
a point $x\in X$ such that $\im(\alpha)\times\{x\}\subseteq A$ or $\{x\}\times\im(\alpha)\subseteq A$, then 
$A$ does not admit a motion plan in $X$.
\end{lemma}

\begin{proof}
Let $\alpha:S^1\to X$ be an essential loop and $x\in X$ be such that $\im(\alpha)\times\{x\}\subseteq A$. Suppose to obtain a contradiction that $A$ admits a motion plan in $X$. Let $\gamma:S^1\to A$ be given by $\gamma(t)=(\alpha(t),x)$. Since the projections $p_1,p_2:A\to X$ are homotopic, the loops $f_1=p_1\circ \gamma:S^1\to X$ and $f_2=p_2\circ \gamma: S^1\to X$ are homotopic. However, $f_1=\alpha$ is essential and $f_2$ is the constant map at $x$; a contradiction. The other case follows using a symmetric argument.
\end{proof}

\begin{lemma}\label{tcobstructionlemma}
Suppose $X$ is a one-dimensional Hausdorff space and $A\subseteq B\subseteq X\times X$ are closed. If $B\backslash A$ admits a motion plan in $X$, then $\fw(Y\times Z)\subseteq A$ for all $Y\times Z\subseteq B$.
\end{lemma}

\begin{proof}
Suppose that $B\backslash A$ admits a motion plan in $X$. Suppose to the contrary that $Y,Z\subseteq X$ are such that $Y\times Z\subseteq B$ and such that $\fw(Y\times Z)$ meets $B\backslash A$ at a point $(y,z)$. Recall from Lemma \ref{productlemma} that $\fw(Y\times Z)=Y\times \fw(Z)\cup \fw(Y)\times Z$. Then it must be the case that either $y\in \fw(Y)$ or $z\in \fw(Z)$. Suppose without loss of generality that $y\in \fw(Y)$. Since $A$ is closed, we may find open sets $U,V\subseteq X$ such that $(y,z)\in (U\times V)\cap B\subseteq B\backslash A$. Since $y\in\fw(Y)$, we can find a loop $\alpha:S^1\to Y\cap U$ that is essential in $Y$. By Corollary \ref{pioneinjectivecor}, $\alpha$ is essential in $X$. Since $\im(\alpha)\times \{z\}\subseteq (Y\cap U)\times \{z\}\subseteq (U\times V)\cap B\subseteq B\backslash A$, we obtain a contradiction to Lemma \ref{torusobstructionlemma}.
\end{proof}

\begin{lemma}\label{tclowerboundlemma}
If $X$ is a one-dimensional Hausdorff space and $\rk(X)=n<\infty$, then $\tc(X)\geq 2n-2$. Moreover,
\begin{enumerate}
\item if $\fw^{n-1}(X)$ contains a simple closed curve, then $\tc(X)\geq 2n-1$;
\item if $X$ is a Peano continuum and $\fw^{n-1}(X)$ contains two distinct simple closed curves, then $\tc(X)\geq 2n$.
\end{enumerate}
\end{lemma}

\begin{proof}
Suppose $\tc(X)=m$ and $\emptyset=A_{-1}\subseteq A_0\subseteq A_1\subseteq \cdots\subseteq A_{m}=X\times X$ are closed subsets such that $A_{j+1}\backslash A_j$ admits a motion plan in $X$ for all $j$. Note that if $\fw^{i}(X)\times \fw^{j}(X)\subseteq A_{k}$, then Lemmas \ref{productlemma} and \ref{tcobstructionlemma} imply that \[\fw^{i}(X)\times \fw^{j+1}(X)\cup \fw^{i+1}(X)\times \fw^{j}(X)=\fw(\fw^{i}(X)\times \fw^{j}(X))\subseteq A_{k-1}.\] Starting with $X\times X=\fw^{0}(X)\times\fw^{0}(X)\subseteq A_{m}$ and recursively taking the wild set in the first coordinate $(n-1$)-times, we have $\fw^{n-1}(X)\times X\subseteq A_{m-n+1}$. We continue by recursively taking the wild in the second coordinate $(n-1)$-times to obtain $\fw^{n-1}(X)\times \fw^{n-1}(X)\subseteq A_{m-2n+2}$. Since $\fw^{n-1}(X)\times \fw^{n-1}(X)\neq \emptyset$, we have $m-2n+2\geq 0$. Thus $m\geq 2n-2$.

For (1), suppose $\fw^{n-1}(X)$ contains a simple closed curve. Then there exists an essential loop $\alpha:S^1\to \fw^{n-1}(X)$ based at a point $x\in \fw^{n-1}(X)$. By Lemma \ref{onedimwildinclusion}, $\alpha$ is also essential in $X$. Now $\im(\alpha)\times\{x\}\subseteq \fw^{n-1}(X)\times \fw^{n-1}(X)\subseteq A_{m-2n+2}$ and thus by Lemma \ref{torusobstructionlemma}, $A_{m-2n+2}$ cannot admit a motion plan. In particular, $m-2n+2\geq 1$ and thus $m\geq 2n-1$.

For (2), suppose $X$ is a Peano continuum and $\fw^{n-1}(X)$ contains two distinct simple closed curves. Then there exists a (not necessarily connected) finite simplicial graph $G\subseteq \fw^{n-1}(X)$ with at least two cycles. Since $G\times G\subseteq A_{m-2n+2}$, Proposition \ref{a2prop} ensures that $m-2n+2\geq 2$. Thus $m\geq 2n$.
\end{proof}

\begin{lemma}\label{tcupperbound2lemma}
If $X$ is a $\fw$-stable one-dimensional Peano continuum with $\rk(X)=n<\infty$, then
\[\tc(X)\leq \begin{cases}
2n-2, & \text{if }\fw^{n-1}(X)\text{ contains no simple closed curve},\\
2n, & \text{ otherwise}.
\end{cases}\]
\end{lemma}

\begin{proof}
By Lemma \ref{stablfiltrationlemma}, there exists closed sets $\emptyset=F_0\subseteq F_1\subseteq F_2\subseteq \cdots \subseteq F_n=X$ (re-indexing from the lemma by $F_j=Y_{n-j}$) such that
\begin{enumerate}
\item $F_j$ is a finite disjoint union of Peano continua for $1\leq j\leq n$,
\item $F_{j}\backslash F_{j-1}$ is a disjoint union of contractible sets for each $2\leq j\leq n$,
\item $F_1$ deformation retracts onto $\fw^{n-1}(X)$.
\end{enumerate}
Since $F_1$ need not be categorical, we can only conclude that $F_{j}\backslash F_{j-1}$ is categorical in $X$ for all $2\leq j\leq n$.

Since $X$ is $\fw$-stable, $\fw^{n-1}(X)$ is a finite disjoint union of one-dimensional Peano continua (Proposition \ref{stableprop}) with no wild points. Mimicking the proof of the inequality $\cat(X\times X)\leq 2\,\cat(X)$ (see Section \ref{secLScatandTC}), we define $H_k=\bigcup_{i+j=k}F_i\times F_j$ for $k=2,3,4,\dots,2n$ and $1\leq i,j\leq n$. For $2\leq k\leq 2n-1$, $H_{k+1}\backslash H_k$ is the disjoint union $\bigcup_{j=1}^{k}((F_{j}\backslash F_{j-1})\times (F_{k-j+1}\backslash F_{k-j}))$ where $(F_{j}\backslash F_{j-1})\times (F_{k-j+1}\backslash F_{k-j})$ is categorical in $X\times X$ whenever $2\leq j\leq k-1$ (and thus admits a motion plan in $X$). However, when $j\in\{1,k\}$ it may not be the case that the sets $F_1\times (F_{k}\backslash F_{k-1})$ and $(F_{k}\backslash F_{k-1})\times F_1$ are categorical in $X\times X$.

\textbf{Case (1):} If $\fw^{n-1}(X)$ admits no simple closed curve, then both $\fw^{n-1}(X)$ and $F_1$ are finite disjoint unions of dendrites. In this case, $F_1$ is categorical in $X$ and it follows that each of the sets $H_2,H_{3}\backslash H_2,H_{4}\backslash H_{3},\dots H_{2n}\backslash H_{2n-1}$ is categorical in $X\times X$ and therefore admits a motion plan in $X$. Hence, the filtration $H_2\subseteq H_3\subseteq \cdots \subseteq H_{2n}=X\times X$ gives $\tc(X)\leq 2n-2$.

\textbf{Case (2):} In the general case, we extend the filtration from Case (1). Since $X$ is $\fw$-stable, $F_1$ is a finite disjoint union of Peano continua $C_1,C_2,\dots,C_r$. Since $F_1$ is homotopy equivalent to $\fw^{n-1}(X)$, and $\fw^{n}(X)=\emptyset$, we have $\fw(F_1)=\emptyset$. Thus each $C_m$ is a one-dimensional Peano continuum containing no wild points. Theorem \ref{graphtheorem} ensures that each $C_m$ strongly deformations retracts onto a (possibly degenerate) finite graph $G_m\subseteq C_m$ where $G_m$ has no free edges. In particular, we take $G_m$ to contain a single point if $C_m$ is a dendrite. Thus $F_1$ strongly deformation retracts onto $G=\bigcup_{m=1}^{r}G_m$. Take $\phi:F_1\to G$ to be a retraction and a homotopy equivalence. Then $\phi\times\phi:F_1\times F_1\to G\times G=\coprod_{1\leq l,m\leq r}G_l\times G_m$ is also a retraction and a homotopy equivalence. Thus, if a set $A\subseteq G\times G$ admits a motion plan in $G$, then $(\phi\times \phi)^{-1}(A)$ admits a motion plan in $F_1$ (and thus in $X$). 

For each $1\leq m\leq r$, let $T_m$ be a maximal spanning tree in $G_m$. For each pair $(l,m)$, let $K^{l,m}_{0}=T_l\times T_m$, $K^{l,m}_{1}=G_l\times T_m\cup T_l\times G_m$, and $K^{l,m}_{2}=G_l\times G_m$. Note that the sets $K^{l,m}_{0}$, $K^{l,m}_{1}\backslash K^{l,m}_{0}$, and $K^{l,m}_{2}\backslash K^{l,m}_{1}$ (for $l\neq m$) are all categorical in $G\times G$ and therefore admit a motion plan in $G$. Define $H_0=(\phi\times \phi)^{-1}(\bigcup_{1\leq l,m\leq r}K^{l,m}_{0})$ and $H_1=(\phi\times \phi)^{-1}(\bigcup_{1\leq l,m\leq r}K^{l,m}_{1})$. Recall that $H_2=F_1\times F_1=(\phi\times \phi)^{-1}(G\times G)$.

Then $H_0\subseteq H_1\subseteq H_2$ where $H_0$, $H_1\backslash H_0$, and $H_2\backslash H_1$ admit motion plans in $X$ since they are finite disjoint unions of sets that admit motion plans in $X$. Hence, the filtration $H_0\subseteq H_1\subseteq H_2\subseteq H_3\subseteq \cdots \subseteq H_{2n}=X\times X$ gives $\tc(X)\leq 2n$.
\end{proof}

\begin{proof}[Proof of Theorem \ref{mainthm4}]
The theorem follows directly by combining the lower bounds in Lemma \ref{tclowerboundlemma} and the upper bounds in Lemma \ref{tcupperbound2lemma}.
\end{proof}

\begin{remark}
As indicated in the introduction, the authors suspect that Lemma \ref{tcupperbound2lemma} can be strengthened to include the improved upper bound $\tc(X)\leq 2n-1$ when $\fw^{n-1}(X)$ contains a single simple closed curve. Unfortunately, our methods used to prove Lemma \ref{tcupperbound2lemma} break down in this case. In Example \ref{wildcircleexample} below, we show that odd values $\tc(X)=2n-1\geq 3$ can occur. However, our construction of a motion plan in this example is ad-hoc and does not easily generalize.
\end{remark}

\begin{example}
Consider the Peano continuum $X$ illustrated in Figure \ref{fig:rankexample}. We have $\rk(X)=3$ where $\fw^2(X)=\ast$. Since $\fw^2(X)$ does not contain a simple closed curve, we have $\cat(X)=2$ and $\tc(X)=2(3)-2=4$.
\end{example}

\begin{example}
It is possible for $\tc(X)$ to achieve the maximum value of $2n$ in Theorem \ref{mainthm4} (when $\rk(X)=n<\infty$). This occurs when $\fw^{n-1}(X)$, which is a finite union of one-dimensional Peano continua, contains more than two cycles. Such a space $X$ can be constructed as follows. Let $n>1$ and $X_1$ be a (not necessarily connected) non-empty finite graph with more than one cycle. Since $X_1$ is a compact metrizable space, \cite[Theorem 1.2]{BrazasMitra} gives the existence of a one-dimensional Peano continuum $X_2$ for which $\fw(X_2)=X_1$. Proceeding as in the first paragraph of Example \ref{nonstableexample}, we may recursively apply \cite[Theorem 1.2]{BrazasMitra} to obtain one-dimensional Peano continua $\emptyset=X_0\subseteq X_1\subseteq X_2\subseteq \cdots \subseteq X_{n}=X$ such that $\fw(X_k)=X_{k-1}$ for all $1\leq k\leq n$. By construction, $X$ is $\fw$-stable. In this case, we have $\rk(X)=n$ and $\fw^{n-1}(X)=X_1$. Since $X_1$ contains more than one cycle, we have $\tc(X)=2n$. See Figure \ref{fig:maxexample} for a planar example in the case $n=2$. \end{example}

\begin{figure}[h]
    \centering
    \includegraphics[scale=0.7]{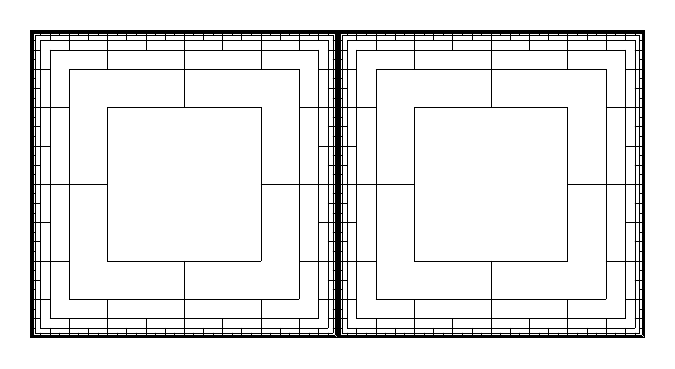}  
    \caption{A one-dimensional Peano continuum $X$ whose wild set $\fw(X)$ is a graph with two cycles. Since $\rk(X)=2$, we have $\tc(X)=4$ by Theorem \ref{mainthm4}.
}
    \label{fig:maxexample}
\end{figure} 

In our next two examples, we show that it is possible to use an ad-hoc construction to give an exact value for $\tc(X)$ in the indeterminate (middle) case of Theorem \ref{mainthm4}.

\begin{example}\label{wildcircleexample}
Consider the planar Peano continuum $X$ illustrated in Figure \ref{fig:wildcircle}. For simplicity, let $C=\fw(X)$ denote the central circle and let $C_1,C_2,C_3,\dots$ be an enumeration of the remaining circles attached to $C$ with intersection $C\cap C_i=\{c_i\}$. Since $\rk(X)=2$, we have $\cat(X)=2$ by Theorem \ref{mainthm2}. However, since $\fw(X)$ contains exactly one simple closed curve, Theorem \ref{mainthm4} gives the indeterminate value $\tc(X)\in\{3,4\}$. We will show that $\tc(X)=3$ by constructing an explicit motion planning filtration.

Since $C_i$ is a circle, we take $A_i\subseteq C_{i}^{2}$ to be the set of antipodal pairs of points in $C_i$. Then $A_i$ admits a motion plan by counterclockwise rotation and $C_{i}^{2}\backslash A_i$ is homeomorphic to the open unit disk. Similarly, we take $A$ to be the set of antipodal pairs of points in $C$ with a counterclockwise motion plan. Let $\Delta_{C}=\{(x,x)\mid x\in C\}$ denote the diagonal of the wild set. Define
\begin{enumerate}
\item $F_0=\Delta_{C}\cup A\cup \bigcup_{i\in\bbn}A_i$,
\item $F_1= C^2\cup \bigcup_{i\in \bbn}C_{i}^{2}$,
\item $F_2=F_1\cup\bigcup_{i\in\bbn}C\times C_i\cup C_i\times C$,
\item $F_3=X\times X$.
\end{enumerate}
Note that $F_0$ and $F_1$ are closed since $\Delta_{C}$ is closed and contains all missing limit points of $\bigcup_{i\in\bbn}A_i$ (and $\bigcup_{i\in\bbn}C_{i}^2$). Similarly, $F_2$ is closed since missing limit points of $\bigcup_{i\in\bbn}C\times C_i\cup C_i\times C$ are contained in $C^2$.

The set $F_0$ admits a motion plan by restricting to the indicated motion plans for $A$ and each $A_i$. Diagonal points $(x,x)\in \Delta_{C}$ are assigned to the constant path at $x$. The set $F_1\backslash F_0$ consists of pairs $(x,y)$ where $x$ and $y$ are non-antipodal points in one of the circles $C,C_1,C_2,C_3,\dots$. We assign such a pair to the standard geodesic path from $x$ to $y$. This defines a motion plan on $F_1\backslash F_0$. We remark that while each of the components of $F_1\backslash F_0$ is homeomorphic to an open disk that is open in $F_1$, these disks are not open in $X$ nor are they of null-diameter. One should regard the tori $C_{i}^{2}$ as shrinking in diameter and limiting on the diagonal circle $\Delta$ of $C^2$ as $i\to\infty$.

The set $F_2\backslash F_1$ consists of pairs $(x,y)$ where one of $x$ or $y$ is in some set $C_i\backslash\{c_i\}$ and the other is in $C\backslash \{c_i\}$. If $x\in C\backslash\{c_i\}$ and $y\in C_i\backslash \{c_i\}$, we assign path $\alpha$ which is the counterclockwise path in $C$ from $x$ to $c_i$ concatenated with the counterclockwise path in $C_i$ from $c_i$ to $y$ (with standard parameterizations). To $(y,x)$ we use assign the reverse path $\alpha^{-1}$ of $\alpha$. Lastly, $F_3\backslash F_2$ consists of pairs $(x,y)$ where $x\in C_i\backslash C$ and $y\in C_j\backslash C$ for $i\neq j$. Following the previous case, we assign such a pair to a concatenation of three counterclockwise-moving paths, again with standard parameterizations, from $x$ to $c_i$ in $C_i$, $c_i$ to $c_j$ in $C$, and $c_j$ to $y$ in $C_j$. 

The above filtration gives upper bound $\tc(X)\leq 3$. Thus $\tc(X)=3$. It follows from Meilstrup's classification \cite{Meilstrup} that $X$ is (non-canonically) homotopy equivalent to the space illustrated in Figure \ref{fig:yexample} and so this polygonal model also has topological complexity of $3$. However, constructing motion plans for the polygonal model is much more complicated.
\end{example}

\begin{figure}[ht]
    \centering
    \includegraphics[height=2in]{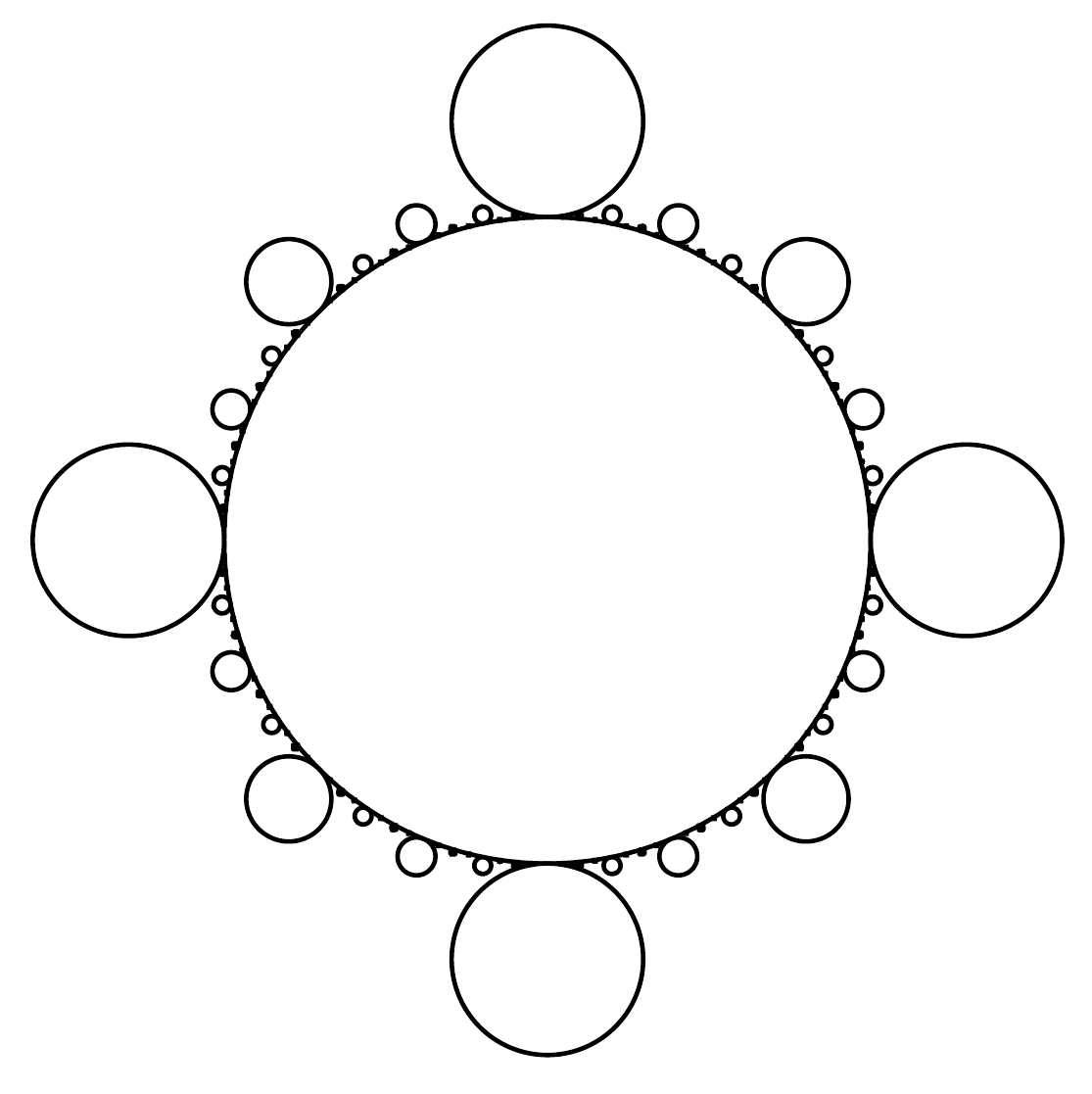}  
    \caption{A one-dimensional planar Peano continuum whose wild set is a circle and for which $\tc(X)=3$.}
    \label{fig:wildcircle}
\end{figure} 

\begin{example}\label{wildwildcircle}
Let $X$ be the planar Peano continuum from Example \ref{wildcircleexample} and recall the notation established in that example. Let $\{d_1,d_2,d_3,\dots\}$ be a set of points in $\bigcup_{i\in\bbn}C_i\backslash C$ that is dense in $X$ and let $D_1,D_2,D_3,\dots$ be a sequence of pairwise-disjoint circles of null-diameter such that $D_j\cap X=\{d_j\}$. For simplicity, we may assume that $D_j\backslash\{d_j\}$ lies in the unbounded component of $\bbr^2\backslash X$. Then $X'=X\cup \bigcup_{j\in\bbn}D_j$ is a one-dimensional planar Peano continuum with $\fw(X')=X$. Since $\rk(X')=3$ and $\fw^2(X')=C$, we have $\tc(X')\in \{5,6\}$ by Theorem \ref{mainthm4}. Let $B_j\subseteq D_{j}^2$ denote the set of antipodal pairs in $D_j$ and let $\Delta_X=\{(x,x)\mid x\in X\}$ be the diagonal of $X$. Define
\begin{enumerate}
\item $G_0=\Delta_X\cup F_0\cup \bigcup_{j\in\bbn}B_j$,
\item $G_1=G_0\cup F_1\cup \bigcup_{j\in\bbn}D_{j}^2$,
\item $G_2=G_1\cup F_2\cup \bigcup_{d_j\in C_i}C_i\times D_j\cup D_j\times C_i$,
\item $G_3=G_2\cup F_3\cup \bigcup_{j\in\bbn}C\times D_j\cup D_j\times C$,
\item $G_4=G_3\cup \bigcup_{d_j\notin C_i}C_i\times D_j\cup D_j\times C_i$,
\item $G_5=X'\times X'$.
\end{enumerate}
Motion plans for the sets $G_0,G_1\backslash G_0,G_2\backslash G_1,\dots, G_5\backslash G_4$ can be constructed in a straightforward way that extends those in Example \ref{wildcircleexample}. In particular, for $2\leq k\leq 5$, the set $G_{k}\backslash G_{k-1}$ contains pairs $(x,y)$ of points from distinct circles (among the circles $C,C_i,D_j$) such that each arc from $x$ to $y$ meets at least $k$ circles. Thus $\tc(X')=5$.
\end{example}

\begin{remark}
The simple geometric nature of the space $X$ from Example \ref{wildcircleexample} and the space $X'$ from Example \ref{wildwildcircle} is what allows us to explicitly construct an optimal motion-planning filtration in each case. From these two examples, there is an evident pattern, which can be used to realize spaces whose topological complexity is an arbitrarily large odd integer. Unfortunately, the idea does not appear to generalize to more complicated one-dimensional Peano continua. The difficulty is compounded in examples of higher wild rank since one cannot change the homeomorphism type of any wild set without changing the homotopy type of the overall space (recall Corollary \ref{hequivonedcor}). 
\end{remark}

\noindent\thanks{{\bf Funding.} This research was supported by the Slovenian Research Agency program P1-0292 and grant J1-4001.}

\end{document}